\documentclass[12pt]{amsart}

\usepackage{amsfonts}
\usepackage{amscd}
\usepackage{amssymb}
\usepackage{enumerate}
\allowdisplaybreaks

\usepackage{hyperref}

\newtheorem{theorem}{Theorem}[section]

\newtheorem{lemma}[theorem]{Lemma}

\newtheorem{proposition}[theorem]{Proposition}
\theoremstyle{definition}

\theoremstyle{remark}

\numberwithin{equation}{section}

\newcommand{\R}{\mathbb{R}}
\newcommand{\G}{\mathcal{G}}
\let\P\undefined
\newcommand{\P}{\mathcal{P}}
\let\H\undefined
\newcommand{\H}{\mathcal{H}}
\let\L\undefined
\newcommand{\L}{\mathcal{L}}

\makeatletter
\renewcommand*{\eqref}[1]{%
  \hyperref[{#1}]{\textup{\tagform@{\ref*{#1}}}}%
}
\makeatother

\title[Hardy's Inequality for the fractional powers of Grushin operator]
{Hardy's Inequality for the fractional powers of the Grushin operator}

\author[R. Balhara]{Rakesh Balhara}
\address[R. Balhara]{Department of Mathematics\\
 Indian Institute of Science\\
560 012 Bangalore, India}

\email{rakeshbalhara@gmail.com}

\keywords{fractional Grushin operator, fractional generalized sublaplacian, Hardy's Inequality, ground state representation, Hecke-Bochner formula}
\subjclass[2010]{Primary: 35A23. Secondary: 26A33, 26D10, 42B37, 42C10, 47A63}
\thanks{Author is financially supported by UGC-CSIR. Also he wants to thank his guide Prof S. Thangavelu for his continuous help and suggestions.}

\begin{document}

\maketitle
\begin{abstract}
We prove Hardy's inequality for the fractional powers of the generalized sublaplacian and the fractional powers of the Grushin operator. We also find an integral representation and a ground state representation for the fractional powers of generalized sublaplacian.
\end{abstract}

\section{Introduction and main results}

The study of various kinds of inequalities for various differential operators are important in understanding many practical problems in physics. Moreover, sharpness of the constants involved in these inequalities is directly related to the existence and nonexistence results for certain partial differential equations.\\

The well known Hardy's inequality for continuously differentiable functions on $\R^{n}$ ($n\geqslant 3$)is given by 
\begin{equation}\label{classical_hardy}
\int_{\R^n}|\nabla f|^2\,dx \geqslant c_n\int_{\R^n}\frac{|f(x)|^2}{|x|^2}\,dx,
\end{equation}
where $\nabla$ is the standard gradient operator on $\R^n$.
Moreover, sharp value of the constant $c_n$, involved in the inequality is known to be equal to $\frac{(n-2)^2}{4}$. By sharp value of the constant, we mean that the inequality will not hold true if we take value of $c_n>\frac{(n-2)^2}{4}$.\\ 

We recall that the classical Laplacian $\Delta$ on $\R^n$ is defined by $\Delta = -\sum_{j=1}^n \frac{\partial^2}{\partial x_j^2}$. For $f\in L^2(\R^n)$ such that $\Delta f\in L^2(\R^n)$, the inequality \eqref{classical_hardy} can be shown equivalent to the inequality:
\begin{equation}
\langle \Delta f,f\rangle\geqslant \frac{(n-2)^2}{4} \int_\R\frac{|f(x)|^2}{|x|^2}\,dx.
\end{equation}  

Hardy's inequality has been generalized for fractional powers of Laplacian. Recall that fractional powers of Laplacian $\Delta^s$, for $s>0$ is defined via spectral decomposition (or Fourier transform) as
\begin{equation}
\Delta^s f(x) = \frac{1}{(2\pi)^n}\int_{\R^n}|\xi|^{2s}\widehat{f}(\xi)e^{ix\cdot\xi}\,d\xi,
\end{equation}
where $\widehat{f}(\xi)$ is the Fourier transform defined by $\int_{\R^n} f(x)e^{-ix\cdot\xi}\,dx$. With this definition, we state the Hardy's inequality for the fractional powers of Laplacian. For $f\in L^2(\R^n)$ such that $\Delta^sf\in L^2(\R^n)$, we have for $0<s<1$
\begin{equation}\label{fractional_Hardy}
\langle \Delta^s f,f\rangle \geqslant 4^s\left(\frac{\Gamma(\frac{n+2s}{4})}{\Gamma(\frac{n-2s}{4})}\right)^2\int_{\R^n}\frac{|f(x)|^2}{|x|^{2s}}\,dx.
\end{equation}
Though the constant involved in the inequality is sharp, the equality is never achieved for any non zero function. Hardy's inequality for the fractional powers of Laplacian has been extensively studied in literature. We refer to \cite{beckner}\cite{franklieb}\cite{herbst}\cite{yafaev} for more details.\\

On the other hand, there is another version of Hardy's inequality for fractional powers of Laplacian, where the homogeneous weight $|x|^{2s}$ has been replaced by a non-homogeneous weight $(\delta+|x|^2)^{2s}$, $\delta>0$:
\begin{equation}\label{hardyTypeLaplacian}
\langle \Delta^{s}f,f\rangle\geqslant (4\delta)^s\left(\frac{\Gamma(\frac{n+2s}{2})}{\Gamma(\frac{n-2s}{2})} \right)^2\int_{\R^n}\frac{|f(x)|^2}{(\delta+|x|^2)^{2s}}\,dx.
\end{equation}
The constant in the inequality is sharp and equality is achieved for $f(x) = (\delta+|x|^2)^{-(n-2s)/2}$. Though the inequality \eqref{hardyTypeLaplacian} is well known, we are unable to find a reference where this inequality is actually proved.\\

In this article, we are interested in proving a similar inequality for the fractional powers of Grushin operator. Recall that the Grushin operator $\G$ on $\R^{n+1}$ is defined by 
\begin{equation}
\G = -\frac{1}{2}\left(\sum_{j=1}^n \frac{\partial^2}{\partial x_j^2}+|x|^2\frac{\partial^2}{\partial w^2}\right).
\end{equation}
Using the spectral decomposition, we can define fractional powers of Grushin operator $\G^s$ for any $s>0$ by 
\begin{equation}
\G^s f(x,w) = \frac{1}{2\pi}\int_{\R}\sum_{k=0}^\infty ((k+n/2)|\lambda|)^s \P_k(\lambda)f^\lambda(x)e^{-i\lambda w}\,d\lambda,
\end{equation}
where $\P_k(\lambda)$ are the orthogonal projections of $L^2(\R^n)$ onto the eigenspaces $E_k^\lambda$ corresponding to the eigenvalues $(2k+n)|\lambda|$ of the scaled Hermite operator $\H(\lambda)$ defined on $\R^n$ as 
\begin{equation}
\H(\lambda) = -\sum_{j=1}^n\frac{\partial^2}{\partial x_j^2}+\lambda^2|x|^2,
\end{equation}
and $f^\lambda$ is the inverse Fourier transform of $f$ in the the last variable, that is
\begin{equation}
f^\lambda(x) = \int_{\R}f(x,w)e^{iw\lambda}\,dw.
\end{equation}
However, it is convenient to work with the following modified fractional powers:
\begin{equation}
\G_s f(x,w)= \frac{1}{2\pi}\int_\R\sum_{k=0}^\infty (2|\lambda|)^s\frac{\Gamma(\frac{2k+n}{4}+\frac{1+s}{2})}{\Gamma(\frac{2k+n}{4}+\frac{1-s}{2})} \P_k(\lambda)f^\lambda(x)e^{-i\lambda w}\,d\lambda.
\end{equation}
Notice that $\G^s$ differs from $\G_s$ by a bounded operator, that is there exists a bounded operator $V_s$ such that $\G^s = V_s \G_s$, which justifies the proving of Hardy type inequality for $\G_s$.\\

We denote by $W^{s,2}(\R^{n+1})$, the Sobolev space consisting of all $L^2(\R^{n+1})$ functions such that $\G_s f\in L^2(\R^{n+1})$. The main theorem that we will prove in this article is the following:
\begin{theorem}\label{hardyInequalityGrushin}
For $f\in W^{s,2}(\R^{n+1})$, $0<s<1$ and $\delta>0$, we have
\[\langle\G_s f,f\rangle\geqslant(4\delta)^s\left(\frac{\Gamma(\frac{n/2+s+1}{2})}{\Gamma(\frac{n/2-s+1}{2})}\right)^2\int_{\R^{n+1}} \frac{|f(x,w)|^2}{\left(\left(\delta+\frac{|x|^2}{2}\right)^2+w^2\right)^s}\,dx\,dw.\] Also the constant in the inequality is sharp and equality is achieved for \[f(x,w) = \left(\left(\delta+\frac{|x|^2}{2}\right)^2+w^2\right)^{-\frac{n/2-s+1}{2}}.\] 
\end{theorem}
In order to prove the above theorem, we prove an analogous theorem for the fractional powers of generalized sublaplacian. We define generalized sublaplacian $\L$ on $\R^+\times \R$ for $\alpha >-1/2$ by
\begin{equation}
\L = -\frac{1}{2}\left(\frac{\partial^2}{\partial x^2}+\frac{2\alpha+1}{x}\frac{\partial}{\partial x}+x^2\frac{\partial^2}{\partial w^2}\right).
\end{equation}
Using the spectral decomposition, we define fractional powers of generalized sublaplacian $\L^s$ for $s>0$ as
\begin{equation}
\L^s f(x,w) = \frac{1}{\pi\Gamma(\alpha+1)}\int_\R\sum\limits_{k=0}^\infty  \left((|\lambda|(2k+\alpha+1))^s \widehat{f}(\lambda, k)\phi^\alpha_{k,\lambda}(x)\right)|\lambda|^{\alpha+1}e^{-i\lambda w}\,d\lambda,
\end{equation} 
where $\widehat{f}(\lambda,k)$ is the Laguerre transform defined as \[\widehat{f}(\lambda,k) = \frac{\Gamma(\alpha+1)\Gamma(k+1)}{\Gamma(\alpha+k+1)}\int_0^\infty\left(\int_R f(x,w)e^{i\lambda w}\,dw\right)\phi^\alpha_{k,\lambda}(x)\,x^{2\alpha+1}\,dx,\] and $\phi^\alpha_{k,\lambda}$ defined as \[\phi_{k,\lambda}^\alpha(x) = L_{k}^\alpha(|\lambda|x^2)e^{-\frac{1}{2}|\lambda|x^2},\] with $L_{k}^\alpha$ denoting the Laguerre polynomials of order $\alpha$.
However, it is convenient to work with the following modified fractional powers of $\L$. For $s>0$, we define $\L_s$ by 
\begin{equation}
\L_sf(x,w) = \frac{1}{\pi\Gamma(\alpha+1)}\int_\R\sum\limits_{k=0}^\infty  \left((2|\lambda|)^s\frac{\Gamma(\frac{2k+\alpha+1}{2}+\frac{1+s}{2})}{\Gamma(\frac{2k+\alpha+1}{2}+\frac{1-s}{2})} \widehat{f}(\lambda, k)\phi^\alpha_{k,\lambda}(x)\right)|\lambda|^{\alpha+1}e^{-i\lambda w}\,d\lambda.
\end{equation}
Again, we denote by $W^{s,2}(\R^+\times\R)$, the space consisting of all functions $f$ such that both $f$ and $\L_s f$ belong to $L^2(\R^+\times\R)$. We will prove the following Hardy type inequality for $\L_s$.
\begin{theorem}\label{hardygeneralizedSublaplacian}
For $f\in W^{s,2}(\R^+\times\R)$, $0<s<1$ and $\delta>0$, we have  \[\langle\L_s f, f\rangle\geqslant(4\delta)^s\left(\frac{\Gamma(\frac{\alpha+s+2}{2})}{\Gamma(\frac{\alpha-s+2}{2})}\right)^2\int_{\R^+}\int_\R \frac{|f(x,w)|^2}{\left(\left(\delta+\frac{x^2}{2}\right)^2+w^2\right)^s}\,dw\,dx.\] Moreover, the constant in the inequality is sharp and equality is achieved for \[f(x,w) = ((\delta + x^2/2)^2+w^2)^{-\frac{\alpha-s+2}{2}}.\]
\end{theorem}

We outline the contents of this paper. In Section 2, we give preliminaries, definitions and facts concerning Laguerre expansions, fractional powers of sublaplacian, fractional powers of Grushin, spherical harmonics and Hecke Bochner formula. In Section 3, we will prove the Hardy type inequality for the fractional powers of sublaplacian. Integral representation and ground state representation for the fractional powers of sublaplacian are also calculated in this Section. In Section 4, Hardy type inequality for the fractional powers of Grushin will be proved. 

\section{Preliminaries}

\subsection{Laguerre expansions on $\R^+$}
Let $\alpha>-1/2$. We equip $Y = \mathbb{R}^+$ with the measure $x^{2\alpha+1}\,dx$, where $dx$ is the standard Lesbegue measure. For $\lambda\in\R\setminus\{0\}$ and $k =0,1,2,\cdots$, we define \textit{Laguerre functions} $\phi_{k,\lambda}^\alpha$ by
\begin{equation}
\phi_{k,\lambda}^\alpha(x) = L_{k}^\alpha(|\lambda|x^2)e^{-\frac{1}{2}|\lambda|x^2},
\end{equation}
where $L_{k}^\alpha$ are the Laguerre polynomials of type $\alpha$. We also define
\begin{equation}
\tilde{\phi}_{k,\lambda}^\alpha(x)=\left(\frac{2\Gamma(k+1)|\lambda|^{\alpha+1}}{\Gamma(\alpha+k+1)}\right)^\frac{1}{2}\phi_{k,\lambda}^\alpha(x).
\end{equation} 

\begin{proposition}\label{orthogonality_of_Laguerre_functions}
For $\lambda\neq 0$, the collection $\{\tilde{\phi}_{k,\lambda}^\alpha(x)\}_{k=0}^\infty$ forms an orthonormal basis for $L^2(Y, x^{2\alpha+1}dx)$.
\end{proposition}

For proof see \cite{thangavelu3} [Proposition 2.4.2].\\

For $x\in Y$, define \textit{Laguerre translation} $T^x_{\alpha,\lambda}$ for functions on $Y$ by 
\begin{equation}
T^x_{\alpha,\lambda}f(y) = \frac{\Gamma(\alpha+1)2^\alpha}{\sqrt{2\pi}}\int_0^\pi f((x^2+y^2+2xy\cos(\theta))^\frac{1}{2})j_{\alpha-\frac{1}{2}}(|\lambda|xy\sin\theta)\sin^{2\alpha}\theta\,d\theta,
\end{equation}
where $j_\alpha(t) = J_\alpha(t)t^{-\alpha}$. Here $J_\alpha$ is the Bessel functions of order $\alpha$. Though the definition of Laguerre translation is quite complicated, but its action on Laguerre functions is simple enough.

\begin{proposition}\label{translation_of_Laguerre_functions}
We have \[T^y_{\alpha,\lambda}\phi^\alpha_{k,\lambda}(x) = \frac{\Gamma(\alpha+1)\Gamma(k+1)}{\Gamma(\alpha+k+1)}\phi^\alpha_{k,\lambda}(y)\phi^\alpha_{k,\lambda}(x).\]
\end{proposition} 

For proof we refer to \cite{thangavelu1}[Theorem 6.1.2].\\

Using Laguerre translation, we define \textit{Laguerre convolution} $f\ast_\lambda g$ for functions $f,g\in L^1(Y,x^{2\alpha+1}dx)$ as 
\begin{equation}
f\ast_\lambda g(x) = \int_0^\infty T^y_{\alpha,\lambda} f(x)g(y)y^{2\alpha+1}\,dy.
\end{equation}

We quickly recall Hilbert space theory for $L^2(Y,x^{2\alpha+1}\,dx)$. Since $\{\tilde{\phi}_{k,\lambda}^\alpha\}_{k=0}^\infty$ forms an orthonormal basis for $L^2(Y,x^{2\alpha+1}\,dx)$, we have $f = \sum_{k=0}^\infty \langle f,\tilde{\phi}_{k,\lambda}^\alpha\rangle\tilde{\phi}_{k,\lambda}^\alpha$ in $L^2$ norm. 
So, for $k = 0,1,2,\dots$, if we also define \textit{Laguerre coefficients} $\hat{f}(k)$ for the function $f\in L^2(Y,x^{2\alpha+1}\,dx)$ by
\begin{equation}
\hat{f}(k) = \frac{\Gamma(\alpha+1)\Gamma(k+1)}{\Gamma(\alpha+k+1)}\int_Y f(x)\phi^\alpha_{k,\lambda}(x)x^{2\alpha+1}\,dx,
\end{equation}  
then we have
\begin{equation}
f = \frac{2|\lambda|^{\alpha+1}}{\Gamma(\alpha+1)}\sum_{k=0}^\infty \hat{f}(k)\phi^\alpha_{k,\lambda}
\end{equation}
in $L^2(Y,x^{2\alpha+1}\,dx)$ norm. Moreover, using Proposition \ref{translation_of_Laguerre_functions} one can check that 
\begin{equation}\label{relation_convolution_transform}
f\ast_\lambda\phi^\alpha_{k,\lambda} = \hat{f}(k)\phi^\alpha_{k,\lambda}.
\end{equation}
Hence we have another representation for $f\in L^2(Y,x^{2\alpha+1}\,dx)$:
\begin{proposition}
For $f\in L^2(Y,x^{2\alpha+1}\,dx)$, we have \[ f = \frac{2|\lambda|^{\alpha+1}}{\Gamma(\alpha+1)}\sum\limits_{k=0}^\infty f\ast_\lambda\phi^\alpha_{k,\lambda}\] in $L^2$ norm.
\end{proposition} 
Again using Proposition \ref{translation_of_Laguerre_functions}, one can check that $\phi_{k,\lambda}^\alpha\ast_\lambda \phi_{j,\lambda}^\alpha(x) = \frac{\Gamma(\alpha+1)}{2|\lambda|^{\alpha+1}}\phi_{k,\lambda}^\alpha(x)\delta_{k,j}$, where $\delta_{k,j}$ is the kronecker delta function. Using this we can easily calculate Laguerre coefficients of $f\ast_\lambda g$ for $f,g\in L^2(Y,x^{2\alpha+1})$.
\begin{proposition}\label{convolution into product}
For $f,g\in L^2(Y,x^{2\alpha+1})$, Laguerre coefficients of $f\ast_\lambda g$ are related to Laguerre coefficients of $f$ and $g$ by \[(\widehat{f\ast_\lambda g})(k) = \hat{f}(k)\hat{g}(k).\]
\end{proposition}

\subsection{Laguerre transform on $\R^+\times\R$}
Let $X = \R^+\times \R$ and $\alpha>-1/2$. We will denote the elements of $X$ by Greek letters $\xi$, $\eta$, etc, with the understanding that $\xi = (x,w)$ means $x\in \mathbb{R}^{+}$ and $w\in\mathbb{R}$. We equip $X$ with measure $d\mu(x,w) = x^{2\alpha+1}dx\,dw$, where $dx$ and $dw$ are the standard Lesbegue measures. For $\xi = (x,w)$ and $\eta = (y,v)$, and $\theta, \phi\in\mathbb{R}$, we define product \[(\xi,\eta)_{\theta,\phi} = ((x^2+y^2-2xy\cos(\theta))^{\frac{1}{2}},w-v+xy\cos(\phi)\sin(\theta))\] Define measure $\nu$ on $[0,\pi]\times [0,\pi]$ by:\[d\nu(\theta,\phi) = \frac{\alpha}{\pi}(\sin(\phi))^{2\alpha-1}(\sin(\theta))^{2\alpha}d\theta\,d\phi,\] where $d\theta$ and $d\phi$ are again standard Lesbegue measures. For $f\in L^1(X,\mu)$, we define \textit{generalized translation operator} $T^\eta$ for $\eta\in X$ by 
\begin{equation}
T^\eta f(\xi) = \int_0^\pi\int_0^\pi f((\xi,\eta)_{(\theta,\phi)})\,d\nu(\theta,\phi).
\end{equation}
Using the generalized translation operator, we define convolution $f\ast g$, for $f,g\in L^1(X,\mu)$ by
\begin{equation}
f\ast g(\xi) =\int_XT^\eta f(\xi)g(\eta)\,d\mu(\eta).
\end{equation}
For $\lambda\in\R\setminus\{0\}$ and $k=0,1,2,\cdots$, define 
\begin{equation}
\psi^\alpha_{k,\lambda}(x,w) = \frac{\Gamma(k+1)\Gamma(\alpha+1)}{\Gamma(\alpha+k+1)}e^{i\lambda w}\phi^\alpha_{k,\lambda}(x).
\end{equation} Though the definition of generalized translation is complicated, its action on $\psi^\alpha_{k,\lambda}$ is simple.
\begin{proposition}\label{generalized translation of laguerre functions}
We have \[T^{(y,-v)}\psi^\alpha_{k,\lambda}(x,w) =\psi^\alpha_{k,\lambda}(y,v)\psi^\alpha_{k,\lambda}(x,w). \]
\end{proposition}
We refer to \cite{stempak1}[Lemma 4.2] for the proof.\\

For $f\in L^1(X)$, $\lambda\in\R\setminus\{0\}$ and $k = 0,1,2,\dots$, we define its \textit{Laguerre transform} $\widehat{f}(\lambda,k)$ by
\begin{equation}
\widehat{f}(\lambda,k) = \int_X f(x,w)\psi^\alpha_{k,\lambda}(x,w)\,d\mu(x,w).
\end{equation}

\begin{proposition}\label{inverse_theorem}
For $f\in L^2(X)$, we have \[f(x,w) = \frac{1}{\pi\Gamma(\alpha+1)} \int_\R\sum\limits_{k=0}^\infty \widehat{f}(\lambda,k)\phi^\alpha_{k,\lambda}(x)|\lambda|^{\alpha+1}e^{-i\lambda w}\,d\lambda,\] in $L^2(X)$ norm.
\end{proposition}
For proof see \cite{stempak2}[Lemma 3.1].\\

For $f\in L^1(X,\mu)$, we define \begin{equation}
f^\lambda(x) = \int_\R f(x,w)e^{i\lambda w}\,dw,
\end{equation} that is inverse Fourier transform in the second variable. For $f\in L^2(X)$, if we calculate the Laguerre coefficients of $f^\lambda$, we immediately notice that 
\begin{equation}
\widehat{f^\lambda}(k) = \widehat{f}(\lambda,k).
\end{equation}
Also using the Proposition \ref{generalized translation of laguerre functions}, one can check that 
\begin{equation}
\widehat{f\ast g} = \widehat{f}\,\widehat{g}.
\end{equation}

\subsection{Fractional powers of generalized sublaplacian}

For $\alpha>-1/2$, we define \textit{generalized sublaplacian} $\L$ on $X$ by
\begin{equation}
-\frac{1}{2}\left(\frac{\partial^2}{\partial x^2}+\frac{2\alpha+1}{x}\frac{\partial}{\partial x}+x^2\frac{\partial^2}{\partial w^2}\right).
\end{equation}
This operator is positive and symmetric in $L^2(X)$. Using the fact that Laguerre polynomials satisfy the following identity
\begin{equation}
x\frac{d^2}{d x^2}L^\alpha_k(x) +(\alpha+1-x)\frac{d}{d x}L^\alpha_k(x)+k L^\alpha_k(x) = 0,
\end{equation}  
one can check that
\begin{equation}\label{eigenvalues_of_sublaplacian}
\L\psi^\alpha_{k,\lambda} = |\lambda|(2k+\alpha+1)\psi^\alpha_{k,\lambda}.
\end{equation} 
Thus $\psi^\alpha_{k,\lambda}$ are the eigenvectors for $\L$ with $|\lambda|(2k+\alpha+1)$ as corresponding eigenvalues. Moreover, for $f\in L^2(X)$ such that $\L f \in L^2(X)$, we have  
\begin{equation}
\widehat{\L f}(\lambda,k) = |\lambda|(2k+\alpha+1)\widehat{f}(\lambda,k).
\end{equation}

Using Proposition \ref{inverse_theorem} and \eqref{eigenvalues_of_sublaplacian}, we obtain the following spectral decomposition of $\L$:
\begin{equation}
\mathcal{L}f(x,w) = \frac{1}{\pi\Gamma(\alpha+1)}\int_\R\left(\sum\limits_{k=0}^\infty|\lambda|(2k+\alpha+1)\; \widehat{f}(\lambda, k)\phi^\alpha_{k,\lambda}(x)\right)|\lambda|^{\alpha+1}e^{-i\lambda w}\,d\lambda.
\end{equation}

Therefore, using spectral decomposition, we define fractional powers of the generalized sublaplacian $\L^s$ for $0<s<1$:
\begin{equation}
\mathcal{L}^s f(x,w) = \frac{1}{\pi\Gamma(\alpha+1)}\int_\R\left(\sum\limits_{k=0}^\infty (|\lambda|(2k+\alpha+1))^s\; \widehat{f}(\lambda, k)\phi^\alpha_{k,\lambda}(x)\right)|\lambda|^{\alpha+1}e^{-i\lambda w}\,d\lambda.
\end{equation} However, it is convenient to work with the following modified fractional power of $\mathcal{L}$. For $0<s<1$, we define $\mathcal{L}_s$ by
\begin{equation}
\L_sf(x,w) = \frac{1}{\pi\Gamma(\alpha+1)}\int_\R\left(\sum\limits_{k=0}^\infty (2|\lambda|)^s\frac{\Gamma(\frac{2k+\alpha+1}{2}+\frac{1+s}{2})}{\Gamma(\frac{2k+\alpha+1}{2}+\frac{1-s}{2})} \; \widehat{f}(\lambda, k)\phi^\alpha_{k,\lambda}(x)\right)|\lambda|^{\alpha+1}e^{-i\lambda w}\,d\lambda.
\end{equation} 
Thus $\mathcal{L}_s$ corresponds to the spectral multiplier \begin{equation}\label{spectralmultiplier}
(2|\lambda|)^s\frac{\Gamma\left(\frac{2k+\alpha+1}{2}+\frac{1+s}{2}\right)}{\Gamma\left(\frac{2k+\alpha+1}{2}+\frac{1-s}{2}\right)}.
\end{equation}
Finally, we define $W^{s,2}(X)$ as the space consisting of $f\in L^2(X)$ such that $\L_sf\in L^2(X)$.

\subsection{Heat semigroup associated with $\L$}
The \textit{heat semigroup} $e^{-t\L}$ generated by $\mathcal{L}$ is defined by the relation
\begin{equation}
\widehat{e^{-t\L}f}(\lambda,k) = e^{-|\lambda|(2k+\alpha+1)t}\widehat{f}(\lambda,k).  
\end{equation}
Thus we have 
\begin{equation}
e^{-t\L} f(x,w) = f\ast h_t(x,w),
\end{equation} where $h_t$ is the \textit{heat kernel} associated with $\mathcal{L}$ given by \begin{equation}
\widehat{h}_t(\lambda,k) = e^{-|\lambda|(2k+\alpha+1)t}.
\end{equation} Although the expression for $h_t(x,w)$ is not known explicitly, we have the explicit expression for $h^\lambda_t(x)$.

\begin{proposition}\label{heat_kernel_expression}
We have \[h^\lambda_t(x) = \frac{2}{\Gamma(\alpha+1)}\left(\frac{\lambda}{2\sinh(\lambda t)}\right)^{\alpha+1}e^{-\frac{\lambda}{2}x^2 \coth(\lambda t)}.\]
\end{proposition}
\begin{proof}
Using Proposition \ref{inverse_theorem}, we have \begin{align*}
h_t(x,w) &= \frac{1}{\pi\Gamma(\alpha+1)}\sum\limits_{k=0}^\infty \int_{-\infty}^\infty e^{-|\lambda|(2k+\alpha+1)t}\phi^\alpha_{k,\lambda}(x)|\lambda|^{\alpha+1}e^{-i\lambda w}\,d\lambda\\
&= \frac{1}{2\pi}\int_{-\infty}^\infty h_t^\lambda(x)e^{-i\lambda w}\,d\lambda,
\end{align*}
where \[h_t^\lambda(x) = \frac{2}{\Gamma(\alpha+1)}\sum\limits_{k=0}^\infty e^{-|\lambda|(2k+\alpha+1)t}\phi^\alpha_{k,\lambda}(x)|\lambda|^{\alpha+1}.\] Using the generating function identity for Laguerre functions \[\sum_{k=0}^\infty L^\alpha_{k}(x)r^k = (1-r)^{-(\alpha+1)}e^{-\frac{rx}{1-r}},\] we simplify to get the desired expression for $h_t^\lambda$.
\end{proof}

For $0<s<1$ and $t>0$, we define $K_{t,s}(x,w)$ by
\begin{equation}\label{modified kernel}
K^\lambda_{t,s}(x) = h^\lambda_{t}(x)\left(\frac{\lambda t}{\sinh \lambda t}\right)^{s+1}.
\end{equation}

\begin{lemma}\label{properties of modified kernel}
We have the following properties of $K_{t,s}$
\begin{equation}\label{eq3}
\int_X K_{t,s}(x,w)\,d\mu(x,w) = 1.
\end{equation}
\begin{equation}\label{eq1}
\int_X T^\eta K_{t,s}(\xi)\,d\mu(\xi) = \int_X T^\eta K_{t,s}(\xi)\,d\mu(\eta) = 1.
\end{equation}
\begin{equation}\label{eq2}
f\ast K_{t,s}(\xi) = K_{t,s}\ast f(\xi).
\end{equation}
\end{lemma}
\begin{proof}
We begin with the definition\[\int_{-\infty}^\infty K_{t,s}(x,w)e^{i\lambda w}\,dw = h^\lambda_{t}(x)\left(\frac{\lambda t}{\sinh \lambda t}\right)^{s+1}.\] Making $\lambda$ go to $0$, we get \[\int_{-\infty}^\infty K_{t,s}(x,w)\,dw = \frac{2}{\Gamma(\alpha+1)}\frac{1}{(2t)^{\alpha+1}}e^{-\frac{x^2}{2t}}.\] Therefore, \begin{align*}\int_X K_{t,s}(x,w)\,d\mu(x,w) &= \int_0^\infty \frac{2}{\Gamma(\alpha+1)}\frac{1}{(2t)^{\alpha+1}}e^{-\frac{x^2}{2t}} x^{2\alpha+1}\,dx\\
&= \int_0^\infty \frac{2}{\Gamma(\alpha+1)}e^{-x^2} x^{2\alpha+1}\,dx\\
&=1.
\end{align*}
Next, using \cite{stempak1}[Lemma 3.1], we have \[\int_X T^{\eta}f(\xi)g(\xi)\,d\mu(\xi) = \int_X f(\xi)T^{\eta^\ast}g(\xi)\,d\mu(\xi),\] where $(x,w)^\ast = (x,-w)$. Take $g = 1$ and $f = K_{t,s}$ and use \eqref{eq3} to conclude \[\int_X T^\eta K_{t,s}(\xi)\,d\mu(\xi) = 1.\] On the other hand, \[\int_X T^\eta K_{t,s}(\xi)\,d\mu(\xi) = \int_X T^{\xi^\ast} K_{t,s}(\eta^{\ast})\,d\mu(\xi)\] is just a change of variables; combining with the fact that $K_{t,s}$ is an even function in the second variable, we conclude second half of the \eqref{eq1}.\\

Finally, \eqref{eq2} follows from the change of variable and the fact that $K_{t,s}$ is an even function in the second variable.
\end{proof}

For $0<s<1$, we define
\begin{equation}
K_s(x,w) = \int_0^\infty K_{t,s}(x,w)t^{-s-1}\,dt.
\end{equation}

We will show $K_{s}$ is a positive function, more precisely:
\begin{proposition}\label{explicit modified kernel}
For $0<s<1$, we have \[K_s(x,w) = \frac{2^{2\alpha+2s+3}\left(\Gamma\Big(\frac{\alpha+s+2}{2}\Big)\right)^2}{\pi\Gamma(\alpha+1)}\frac{1}{(x^4+4w^2)^{\frac{s+\alpha+2}{2}}}.\] 
\end{proposition}
\begin{proof}
Calculations are borrowed from \cite{thangaveluroncal}[Proposition 4.2]. We repeat for the sake of completeness. We start with the expression
$$
\int_{-\infty}^{\infty}K_s(x,w)e^{i\lambda w}\,dw=\int_0^{\infty}h_t^{\lambda}(x)\Big(\frac{t|\lambda|}{\sinh t|\lambda|}\Big)^{s+1}t^{-s-1}\,dt.
$$
Using Proposition \ref{heat_kernel_expression}, and since the functions involved are even in $\lambda$, we have
$$
\int_{-\infty}^{\infty}K_s(x,w)e^{i\lambda w}\,dw=\frac{2}{\Gamma(\alpha+1)}\int_0^{\infty}\Big(\frac{\lambda}{\sinh t\lambda}\Big)^{\alpha+s+2}e^{-\frac 1 2\lambda(\coth t\lambda)x^2}\,dt.
$$
As the Fourier transform of $K_s$ in the central variable $w$ is an even function of $\lambda$ we have, after taking the Fourier transform in the variable $\lambda$,
$$
K_s(x,w)=\frac{2}{\pi\Gamma(\alpha+1)}\int_0^{\infty}\int_0^{\infty}(\cos \lambda w)\Big(\frac{\lambda}{\sinh t\lambda}\Big)^{\alpha+s+2}e^{-\frac12\lambda(\coth t\lambda)x^2}\,d\lambda\,dt.
 $$
By the change of variables $\lambda\to \lambda x^{-2}$, $t\to tx^2$, we obtain
\begin{equation}
\label{eq:homogeneity}
K_s(x,wx^2)=x^{-2(\alpha+s+2)}K_s(1,w).
\end{equation}
Thus
\begin{align*}
K_s(1,w)&=\frac{2}{\pi\Gamma(\alpha+1)}\int_0^{\infty}\int_0^{\infty}(\cos \lambda w)\Big(\frac{\lambda}{\sinh t \lambda}\Big)^{\alpha+s+2}e^{-\frac{\lambda}{2}(\coth t\lambda)}\,dt\,d\lambda\\
&=\frac{2}{\pi\Gamma(\alpha+1)}\int_0^{\infty} \Big( \int_0^{\infty}(\cos \lambda w) \lambda^{\alpha+s+1}e^{-\frac{\lambda}{2}(\coth t)} \,d\lambda \Big) (\sinh t )^{-(\alpha+s+2)} \,dt.
\end{align*}
The integral in $\lambda$ can be evaluated by using \cite[p.~498,
~3.944.6]{GR}:
\begin{equation}
\label{eq:GR1}
\int_0^{\infty}x^{\mu-1}e^{-\beta x}(\cos \delta x)\,dx=\frac{\Gamma(\mu)}{(\delta^2+\beta^2)^{\mu/2}}\cos\Big(\mu\arctan \frac{\delta}{\beta}\Big),
\end{equation}
valid for $\operatorname{Re}\mu>0$, $\operatorname{Re}\beta>|\operatorname{Im}\delta|$. Taking with $\mu=\alpha+s+2$, $\beta=\frac{1}{2}(\coth t)$ and $\delta=w$. Then, we get
$$
\int_0^{\infty}(\cos \lambda w)\lambda^{\alpha+s+1}e^{-\frac{\lambda}{2}(\coth t)}\,d\lambda=\frac{\Gamma(\alpha+s+2) \cos \Big((\alpha+s+2)\arctan\Big(\frac{2w}{\coth t}\Big)\Big)}{\Big(w^2+\frac{1}{4}\coth^2t\Big)^{\frac{\alpha+s+2}{2}}}.
$$
Thus
\begin{equation}
\label{eq:1Hom}
K_s(1,w)= \frac{2\Gamma(\alpha+s+2)}{\pi\Gamma(\alpha+1)}\int_0^{\infty}\frac{\cos \Big((\alpha+s+2)\arctan\Big(\frac{2w}{\coth t}\Big)\Big) }{\Big(w^2+\frac{1}{4}\coth^2t\Big)^{\frac{\alpha+s+2}{2}}}(\sinh t)^{-(\alpha+s+2)}\,dt.
\end{equation}
With the change of variables $u=\frac{2w}{\coth t}$ we have that the latter integral equals
\begin{align*}
\int_0^{2w}&\Big(\frac{u^2}{4w^2-u^2}\Big)^{-\frac{(\alpha+s+2)}{2}}
\Big(w^2+\frac{4w^2}{4u^2}\Big)^{-\frac{(\alpha+s+2)}{2}}
\cos[(\alpha+s+2)\arctan u]\frac{2w}{4w^2-u^2}\,du\\
&=2w^{-(\alpha+s+1)}
\int_0^{2w}(4w^2-u^2)^{\frac{\alpha+s}{2}}(1+u^2)^{-\frac{\alpha+s+2}{2}}\cos[(\alpha+s+2)\arctan u]\,du\\
&=2^{\alpha+s+1}w^{-1}
\int_0^{2w}\Big(1-\frac{u^2}{4w^2}\Big)^{\frac{\alpha+s}{2}}(1+u^2)^{-\frac{\alpha+s+2}{2}}\cos[(\alpha+s+2)\arctan u]\,du.
\end{align*}
Thus, with this and \eqref{eq:1Hom} we have
\begin{equation}
\label{eq:2Hom}
\mathcal{K}_s(1,w)= \frac{2^{\alpha+s+2} \Gamma(\alpha+s+2)}{\pi\Gamma(\alpha+1)}  w^{-1} I,
\end{equation}
where
$$
I:=\int_0^{2w}\Big(1-\frac{u^2}{4w^2}\Big)^{\frac{\alpha+s}{2}}(1+u^2)^{-\frac{\alpha+s+2}{2}}\cos[(\alpha+s+2)\arctan u]\,du.
$$
Now we will see that the above integral can be explicitly computed in terms of Legendre functions. Making a second change of variable $\arctan u=z$, the integral $I$ becomes
$$
I=\int_0^{\arctan2w}\Big(\cos^2z-\frac{\sin^2z}{4w^2}\Big)
^{\frac{\alpha+s}{2}}\cos[(\alpha+s+2)z]\,dz.
$$
We can rewrite the above integral as
\begin{align*}
I&=\int_0^{\arctan2w}\Big(\frac{1+\cos2z}{2}-\frac{1-\cos2z}{2\cdot4w^2}\Big)^{\frac{\alpha+s}{2}}
\cos[(\alpha+s+2)z]\,dz\\
&=2^{-\frac{\alpha+s}{2}}\int_0^{\arctan2w}\bigg((\cos 2z)\Big(1+\frac{1}{4w^2}\Big)-\Big(\frac{1}{4w^2}-1\Big)\bigg)^{\frac{\alpha+s}{2}}
\cos[(\alpha+s+2)z]\,dz\\
&=\Big(\frac{1+4w^2}{8w^2}\Big)^{\frac{\alpha+s}{2}}\int_0^{\arctan2w}\Big(\cos2z-\frac{1-4w^2}{1+4w^2}\Big)^{\frac{\alpha+s}{2}}\cos[(\alpha+s+2)z]\,dz\\
&=\frac12\Big(\frac{1+4w^2}{8w^2}\Big)^{\frac{\alpha+s}{2}}\int_0^{2\arctan2w}(\cos\beta-\cos\gamma)^{\frac{\alpha+s}{2}}\cos\Big[\frac{(\alpha+s+2)}{2}\beta\Big]\,d\beta,
\end{align*}
where $\cos\gamma=\frac{1-16w^2}{1+16w^2}$. The integral can be evaluated using (\cite[p. 406, 3.663.1]{GR}):
\begin{equation}
\label{eq:GR2}
\int_0^u(\cos x-\cos u)^{\nu-\frac12}\cos ax\,dx=\sqrt{\frac{\pi}{2}}(\sin u)^{\nu}\Gamma\Big(\nu+\frac12\Big)P_{a-\frac12}^{-\nu}(\cos u),
\end{equation}
valid for $ \operatorname{Re}\nu>-\frac12$, $a>0$, $0<u<\pi$, where $P_{a-\frac12}^{-\nu}$ is an associated Legendre function of the first kind (see for instance \cite[Sections 8.7-8.8]{GR}). Also Recall the following representation for the associated Legendre function (\cite[p. 969, 8.755]{GR})
\begin{equation}
\label{eq:Legendre}
P_{\nu}^{-\nu}(\cos\varphi)=\frac{\big(\frac{\sin\varphi}{2}\big)^{\nu}}{\Gamma(1+\nu)}.
\end{equation}
Taking $\nu=\frac{\alpha+s+1}{2}$ and $a=\frac{\alpha+s+2}{2}$ in \eqref{eq:GR2} and using the representation for the associated Legendre function \eqref{eq:Legendre}, the latter integral becomes
\begin{align*}
\sqrt{\frac{\pi}{2}}(\sin \gamma)^{\frac{\alpha+s+1}{2}}\Gamma\Big(\frac{\alpha+s+2}{2}\Big)P_{\frac{\alpha+s+1}{2}}^{-\frac{\alpha+s+1}{2}}(\cos \gamma)&=\sqrt{\frac{\pi}{2}}\Gamma\Big(\frac{\alpha+s+2}{2}\Big)(\sin \gamma)^{\frac{\alpha+s+1}{2}}
\frac{(\sin\gamma)^{\frac{\alpha+s+1}{2}}}{2^{\frac{\alpha+s+1}{2}}\Gamma\big(\frac{\alpha+s+3}{2}\big)}\\
&=\sqrt{\frac{\pi}{2}}\frac{\Gamma\big(\frac{\alpha+s+2}{2}\big)}
{2^{\frac{\alpha+s+1}{2}}\Gamma\big(\frac{\alpha+s+3}{2}\big)}(\sin^2\gamma)^{\frac{\alpha+s+1}{2}}\\
&=\sqrt{\frac{\pi}{2}}\frac{\Gamma\big(\frac{\alpha+s+2}{2}\big)}{2^{\frac{\alpha+s+1}{2}
\Gamma\big(\frac{\alpha+s+3}{2}\big)}}\Big(\frac{4w}{1+4w^2}\Big)^{\alpha+s+1},
\end{align*}
because $\sin^2\gamma=\frac{16w^2}{(1+4w^2)^2}$. This gives
\begin{equation}
\label{eq:3Hom}
 I=\frac12  \sqrt{\frac{\pi}{2}}\frac{\Gamma\big(\frac{\alpha+s+2}{2}\big)}{2^{\frac{\alpha+s+1}{2}
\Gamma\big(\frac{\alpha+s+3}{2}\big)}}\Big(\frac{1+4w^2}{8w^2}\Big)^{\frac{\alpha+s}{2}}\Big(\frac{4w}{1+4w^2}\Big)^{\alpha+s+1}= \sqrt{\pi}\frac{\Gamma\big(\frac{\alpha+s+2}{2}\big)}{
\Gamma\big(\frac{\alpha+s+3}{2}\big)}w(1+4w^2)^{-\frac{\alpha+s+2}{2}}.
\end{equation}
Finally, plugging \eqref{eq:3Hom} into  \eqref{eq:2Hom}, we have
$$
K_s(1,w)= \frac{2^{\alpha+s+2} \Gamma(\alpha+s+2)}{\pi\Gamma(\alpha+1)}\sqrt{\pi}\frac{\Gamma\big(\frac{\alpha+s+2}{2}\big)}{
\Gamma\big(\frac{\alpha+s+3}{2}\big)}(1+4w^2)^{-\frac{\alpha+s+2}{2}},
$$
or, by \eqref{eq:homogeneity}
$$
K_s(x,w)=x^{-2(\alpha+s+2)}K_s\Big(1,\frac{w}{x^2}\Big)=c_{\alpha,s} (x^4+4w^2)^{-\frac{\alpha+s+2}{2}}.
$$
where the constant $ c_{\alpha,s} $ is given by
$$
c_{\alpha,s} = \frac{2^{\alpha+s+2} \Gamma(\alpha+s+2)}{\sqrt{\pi}\Gamma(\alpha+1)}\frac{\Gamma\big(\frac{\alpha+s+2}{2}\big)}{
\Gamma\big(\frac{\alpha+s+3}{2}\big)}.
$$
By using Legendre's duplication formula 
\begin{equation}
\label{eq:LegendreDup}
\sqrt{\pi} \Gamma(2z) = 2^{2z-1}\Gamma(z)\Gamma\Big(z+\frac12\Big)
\end{equation}
with $z=\frac{\alpha+s+2}{2}$, and after simplification, we get
$$
c_{\alpha,s} = \frac{2^{2\alpha+2s+3}\left(\Gamma\left(\frac{\alpha+s+2}{2}\right)\right)^2}{\pi\Gamma(\alpha+1)}.
$$
This completes the proof of the Proposition.
\end{proof}

We have shown that $K_{s}$ is a positive function. Moreover, generalized translation is a positive operator, see \cite{stempak1}[Proposition 3.2]. Therefore, $T^\eta K_s\geqslant 0$. In other words,
\begin{equation}\label{positivity}
\int_X T^\eta K_{t,s}(\xi)t^{-s-1}\,dt \geqslant 0,\quad \forall\xi,\eta\in X.
\end{equation}

\subsection{Fractional powers of Grushin operator}
Let $H = \mathbb{R}^n\times\mathbb{R}$ with the understaning that $(x,w)\in H$ means $x\in\R^n$ and $w\in\R$. We equip $H$ with the measure $d\mu(x,w) = dx\,dw$, where $dx$ and $dw$ are the usual Lebesgue measures on $\mathbb{R}^n$ and $\mathbb{R}$. We define \textit{Grushin operator} $\mathcal{G}$ on $H$ by\begin{equation}
\mathcal{G} =-\frac{1}{2}\left(\sum_{j=1}^n\frac{\partial^2}{\partial x_j^2}+|x|^2\frac{\partial^2}{\partial w^2}\right).
\end{equation} For $\lambda\in\R\setminus\{0\}$, we define the\textit{scaled Hermite operator} $\H(\lambda)$ on $\mathbb{R}^n$ by \begin{equation}
\H(\lambda) = -\left(\sum_{j=1}^n\frac{\partial^2}{\partial x_j^2}-\lambda^2|x|^2\right)
\end{equation} For multi-index $\beta\in\mathbb{N}^n$, define $\Phi_{\beta}(x) = h_{\beta_1}(x_1)h_{\beta_2}(x_2)\cdots h_{\beta_{n}}(x_n)$, where $\beta=(\beta_1,\cdots,\beta_n)$ and $h_{\beta_i}$ are normalized Hermite functions. Further, for $\lambda\in\R\setminus\{0\}$ define \[\Phi^\lambda_{\beta}(x) = |\lambda|^{\frac{n}{4}}\Phi_{\beta}(\sqrt{|\lambda|}x).\] The collection $\{\Phi^\lambda_{\beta}\}_{\beta\in\mathbb{N}^n}$ forms an orthonormal basis for $L^2(\mathbb{R}^n)$, see \cite{thangavelu3}[Theorem 1.2.2]. Also they are eigenfunctions for the scaled Hermite operator, that is
\begin{equation}
\H(\lambda)\Phi^\lambda_{\beta} = (2|\beta|+n)|\lambda|\Phi^\lambda_{\beta}.
\end{equation} 
For $\lambda\in\R\setminus\{0\}$ and $k=0,1,2,\cdots$, define $\P_{k}(\lambda)$ as projections of $L^2(\R^n)$ onto $E^\lambda_{k}$, the eigenspace corresponding to eigenvalue $(2k+n)|\lambda|$. In other words, \begin{equation}
\P_{k}(\lambda)f = \sum_{|\beta|= k}\langle f,\Phi^\lambda_{\beta}\rangle \Phi^\lambda_\beta.
\end{equation} Thus we have
\begin{equation}\label{Hermite_spectral_decomposition}
\H(\lambda) = \sum_{k=0}^\infty(2k+n)|\lambda|\P_k(\lambda).
\end{equation} Finally, using the Fourier transform and the above spectral decomposition of scaled Hermite operator (Equation \eqref{Hermite_spectral_decomposition}), we have the following spectral decomposition of $\G$:
\begin{equation}
\G f(x,w)= \frac{1}{2\pi}\int_\R \sum_{k=0}^\infty (k+n/2)|\lambda|\P_k(\lambda)f^\lambda(x)e^{-i\lambda w}\,d\lambda. 
\end{equation}
Therefore, a natural way to define fractional powers of Grushin operator is via spectral decomposition: \begin{equation}
\G^sf(x,w)= \frac{1}{2\pi}\int_\R \sum_{k=0}^\infty ((k+n/2)|\lambda|)^s \P_k(\lambda)f^\lambda(x)e^{-i\lambda w}\,d\lambda.
\end{equation} However, it is convenient to with following modified fractional powers of $\mathcal{G}$. For $0<s<1$, we define $\mathcal{G}_s$ by \begin{equation}
\G_s f(x,w)= \frac{1}{2\pi}\int_\R \sum_{k=0}^\infty (2|\lambda|)^s\frac{\Gamma(\frac{2k+n}{4}+\frac{1+s}{2})}{\Gamma(\frac{2k+n}{4}+\frac{1-s}{2})} \P_k(\lambda)f^\lambda(x)e^{-i\lambda w}\,d\lambda.
\end{equation}
Also we define $W^{s,2}(H)$ as the space consists of all those functions $f$ in $L^2(H)$ such that $\G_sf\in L^2(H)$ too.
\subsection{Spherical harmonics and Hecke-Bochner formula}\label{Spherical_harmonics}
Let us quickly recall some facts about spherical harmonics and solid harmonics. We refer to \cite{stein1}[Chapter 4] for missing details. Let $\mathfrak{H}_m$ denote the space of spherical harmonics of degree $m$. Let $\{Y_{m,j}\}_{j=1}^{a_m}$ denotes the orthogonal basis of $\mathfrak{H}_m$, where $a_m$ denotes the dimension of $\mathfrak{H}_m$. We know that  $L^2(\mathbb{S}^{n-1}) = \bigoplus_{m=0}^\infty \mathfrak{H}_m$ and the collection $\{Y_{m,j}\}$, for $j = 1,2,\cdots,a_m$ and $m= 0,1,2,\cdots$, forms orthonormal basis for $L^2(\mathbb{S}^{n-1})$. Note that $\mathbb{S}^{n-1}$ denotes the unit sphere in $\R^n$. Define solid harmonics $P_{m,j}(x) = |x|^mY_{m,j}(x/|x|)$ for $j = 1,2\cdots, a_m$ and $m=0,1,\cdots$. Let $\mathfrak{h}_m$ denotes the space consisting of linear combination of functions of the form $f(|x|)P(x)$, where $f$ varies over radial functions and $P\in \mathfrak{H}_m$, with the stipulation that each $f(|x|)P(x)\in L^2(\R^n)$. With these definitions, we have $L^2(\R^n) = \bigoplus_{m=0}^\infty \mathfrak{h}_m$. So for $f\in L^2(\R^{n+1})$, we have 
\begin{equation}
f(x,w) = \sum_{m=0}^\infty\sum_{j=1}^{a_m}f_{m,j}(|x|,w)P_{m,j}(x),
\end{equation}
where $f_{m,j}(|x|,w) = \int_{\mathbb{S}^{n-1}}f(|x|\omega,w)P_{m,j}(|x|\omega)\,d\omega$.\\

Finally, we recall Hecke-Bochner formula, which describes how Hermite projections act on solid harmonics, proof of which can be found in the book \cite{thangavelu1}[Theorem 3.4.1].
\begin{proposition}\label{Hermite_projection_heckhe_bockner}
Suppose $f\in L^2(\mathbb{R}^{n})$ is such that $f=gP$, where $g$ is radial and $P$ is a solid harmonics of degree $m$. Then we have \[\P_{2k+m}(\lambda)f(x) = R^\lambda_{k,m}(g)\phi^{\alpha}_{k,\lambda}(|x|)P(x)\] where \[R^\lambda_{k,m}(g) = \frac{2|\lambda|^{\alpha+1}\Gamma(k+1)}{\Gamma(\alpha+k+1)}\int_0^\infty g(s)\phi^\alpha_{k,\lambda}(s)s^{2\alpha+1}\,ds\] and $\alpha = \frac{n}{2}+m-1$. For other values of $j$, $\P_j(\lambda)f = 0$.
\end{proposition}

\section{Hardy's Inequality for Generalized Subaplacian}

For $-1<s<1$ and $\delta>0$, we define
\begin{equation}\label{function u}
u_{s,\delta}(x,w) = \left(\left(\delta+\frac{x^2}{2}\right)^2+w^2\right)^{-\frac{s+\alpha+2}{2}}.
\end{equation}

\begin{proposition}\label{main_proposition}
For $0<s<1$, we have \[\L_su_{-s,\delta}(\xi) = (4\delta)^s \left(\frac{\Gamma(\frac{\alpha+s+2}{2})}{\Gamma(\frac{\alpha-s+2}{2})}\right)^2 u_{s,\delta}(\xi).\]
\end{proposition}
\begin{proof}
We prove the result by calculating Laguerre transform on both sides. This has been already done in Ciaurri et al \cite{thangaveluroncalciaurri}, but we repeat the calculations for the convenience of the readers. Define \[L(a,b,c) = \int_0^\infty e^{-a(2x+1)}x^{b-1}(1+x)^{-c}\,dx.\] We start with the generating function identity for the Laguerre functions: \[\sum_{k=0}^\infty z^k L^\alpha_k(x^2)e^{-\frac{1}{2}x^2} = (1-z)^{-\alpha-1}e^{-\frac{1}{2}\frac{1+z}{1-z}x^2}.\] Therefore, we have 
\begin{equation}\label{modified_generating_identity}
\sum_{k=0}^\infty \left(\frac{y}{y+|\lambda|}\right)^k L^\alpha_k(|\lambda|x^2)e^{-\frac{1}{2}|\lambda|x^2} = |\lambda|^{-\alpha-1}(y+|\lambda|)^{\alpha+1}e^{-\frac{1}{2}(2y+|\lambda|)x^2}.
\end{equation} For functions $f,g$ defined on $(0,\infty)$, let $F, G$ be their Laplace transforms defined by \[F(a+ib) = \int_{0}^\infty e^{-(a+ib)y}f(y)\,dy,\quad G(a,ib) = \int_{0}^\infty e^{-(a+ib)y}g(y)\,dy,\quad a>0,b\in\mathbb{R}\] Let $\beta = \frac{1}{2}(\alpha+s+2)$. Then with $f(y) = g(y) = \Gamma^{-1}(\beta)y^{\beta-1}e^{-\delta y}$, we have \[F(a+ib) = G(a+ib) = (\delta+a+ib)^{-\beta}.\] On the other hand, it can be checked, see [\cite{cowlinghaagerup}, Lemma 3.4], that \[\int_{-\infty}^\infty F(a+ib)\overline{G(a+ib)}e^{-i|\lambda|b}\,db = 2\pi \int_0^\infty f(y)g(y+|\lambda|)e^{-a(2y+|\lambda|)}\,dy.\] Taking $a = \frac{1}{2}x^2$, we have\[\int_{-\infty}^{\infty}\left(\left(\delta+\frac{1}{2}x^2\right)^2+b^2\right)^{-\frac{1}{2}(s+\alpha+2)}e^{-i|\lambda|b}\,db = 2\pi \int_0^\infty f(y)g(y+|\lambda|)e^{-\frac{1}{2}(2y+|\lambda|)x^2}\,dy.\] Since $u_{s,\delta}(x,w)$ is symmetric in $w$ variable, therefore we have \[u^\lambda_{s,\delta}(x) = 2\pi \int_0^\infty f(y)g(y+|\lambda|)e^{-\frac{1}{2}(2y+|\lambda|)x^2}\,dy.\] Using \eqref{modified_generating_identity}, we have \[u^\lambda_{s,\delta}(x) = \frac{2|\lambda|^{\alpha+1}}{\Gamma(\alpha+1)}\sum_{k=0}^\infty c^\lambda_{k,\delta}(s)\phi^\alpha_{k,\lambda}(x),\] where the coefficients are given by\begin{align*}
c^\lambda_{k,\delta}(s)&= \pi\Gamma(\alpha+1)\int_0^\infty f(y)g(y+|\lambda|)(y+|\lambda|)^{-(k+\alpha+1)}y^k\,dy\\
&= \frac{\pi\Gamma(\alpha+1)|\lambda|^s}{(\Gamma(\beta))^2}\int_0^\infty e^{-\delta(2y+|\lambda|)}y^{\beta+k-1}(y+|\lambda|)^{\beta-k-\alpha-2}\,dy\\
&=\frac{\pi\Gamma(\alpha+1)|\lambda|^s}{(\Gamma(\beta))^2}L\left(\delta|\lambda|,\frac{2k+\alpha+2+s}{2},\frac{2k+\alpha+2-s}{2}\right).
\end{align*}
Notice that $\widehat{u_{s,\delta}}(\lambda,k) = \widehat{u_{s,\delta}^\lambda}(k) = c^\lambda_{k,\delta}(s)$. Also, according to [\cite{cowlinghaagerup}, Proposition 3.6] the function $L$ satisfies the following identity\[\frac{(2\lambda)^a}{\Gamma(a)}L(\lambda,a,b) = \frac{(2\lambda)^b}{\Gamma(b)}L(\lambda,b,a)\] for all $(a,b\in\mathbb{C})$ and $\lambda>0$. Using this we get \begin{equation}
c^\lambda_{k,\delta}(-s) = (2\delta)^s|\lambda|^{-s}\left(\frac{\Gamma(\frac{\alpha+s+2}{2})}{\Gamma(\frac{\alpha-s+2}{2})}\right)^2\frac{\Gamma(\frac{2k+\alpha+1}{2}+\frac{1-s}{2})}{\Gamma(\frac{2k+\alpha+1}{2}+\frac{1+s}{2})}c^\lambda_{k,\delta}(s).
\end{equation} Rearranging the terms, \[(2|\lambda|)^s\frac{\Gamma(\frac{2k+\alpha+1}{2}+\frac{1+s}{2})}{\Gamma(\frac{2k+\alpha+1}{2}+\frac{1-s}{2})} c^\lambda_{k,\delta}(-s) = (4\delta)^s\left(\frac{\Gamma(\frac{\alpha+s+2}{2})}{\Gamma(\frac{\alpha-s+2}{2})}\right)^2c^\lambda_{k,\delta}(s),\] which is nothing but \[\widehat{\mathcal{L}_su_{-s,\delta}}(\lambda,k) = (4\delta)^s \left(\frac{\Gamma(\frac{\alpha+s+2}{2})}{\Gamma(\frac{\alpha-s+2}{2})}\right)^2\widehat{u_{s,\delta}}(\lambda,k),\] and hence the result.
\end{proof}

Next, we find an integral representation for $\mathcal{L}_s$ in an analogous way as it has been found for the fractional powers of sublaplacian on Heisenberg group by Roncal et al \cite{thangaveluroncal}[Section 4].
\begin{theorem}\label{integral representation}
For $0<s<1$ and $f\in W^{s,2}(X)$, we have \[\mathcal{L}_s f(\xi) = \frac{1}{|\Gamma(-s)|}\int_0^\infty (f(\xi)-f\ast K_{t,s}(\xi))t^{-s-1}\,dt,\] where $K_{t,s}$ is defined in \eqref{modified kernel}.
\end{theorem}
\begin{proof}
We begin with the identity (see \cite{GR}. p.382, 3.541.1]) 
\[2^{1-s}\int_0^{\infty}e^{-(\mu+1)t}(\sinh t)^{-s}\,dt=\frac{\Gamma(1-s)\Gamma\Big(\frac{\mu}{2}+\frac{1+s}{2}\Big)}{\Gamma\Big(\frac{\mu}{2}+\frac{1-s}{2}+1\Big)},\] which gives 
\begin{equation}\label{equation 1}
(\mu+1-s)\int_0^{\infty}e^{-(\mu+1)t}(\sinh t)^{-s}\,dt=\frac{2^{s}\Gamma(1-s)\Gamma\Big(\frac{\mu}{2}+\frac{1+s}{2}\Big)}{\Gamma\Big(\frac{\mu}{2}+\frac{1-s}{2}\Big)}.
\end{equation} Also we have \begin{align*}
(\mu+1)\int_0^{\infty}e^{-(\mu+1)t}(\sinh t)^{-s}\,dt&=\int_0^{\infty}\frac{d}{dt}(1-e^{-(\mu+1)t})(\sinh t)^{-s}\,dt\\
&=s\int_0^{\infty}(1-e^{-(\mu+1)t})(\sinh t)^{-s-1}(\cosh t)\,dt.
\end{align*} Therefore, plugging the latter into \eqref{equation 1}, we get \begin{align*}
\frac{2^s\Gamma(1-s)\Gamma\Big(\frac{\mu}{2}+\frac{1+s}{2}\Big)} {\Gamma\Big(\frac{\mu}{2}+\frac{1-s}{2}\Big)}&=s\int_0^{\infty}\Big(\cosh t-e^{-(\mu+1)t}(\cosh t+\sinh t)\Big)(\sinh t)^{-s-1}\,dt\\
&=s\int_0^{\infty}\big(\cosh t-e^{-\mu t}\big)(\sinh t)^{-s-1}\,dt\\
&=s\int_0^{\infty}\big(\cosh t-1\big)(\sinh t)^{-s-1}\,dt+s\int_0^{\infty}\big(1-e^{-\mu t}\big)(\sinh t)^{-s-1}\,dt\\
&=c_1 s+s\int_0^{\infty}\big(1-e^{-\mu t}\big)(\sinh t)^{-s-1}\,dt,
\end{align*}
where $c_1$ is the constant given by
\[c_1:=\int_0^{\infty}\big(\cosh t-1\big)(\sinh t)^{-s-1}\,dt.\]
Thus, by taking $\mu=2k+\alpha+1$ and changing $t$ into $|\lambda|t$, we have
\[\frac{2^s\Gamma(1-s)}{s}\frac{\Gamma\Big(\frac{2k+\alpha+1}{2}+\frac{1+s}{2}\Big)}
{\Gamma\Big(\frac{2k+\alpha+1}{2}+\frac{1-s}{2}\Big)}=c_1 +|\lambda|\int_0^{\infty}\big(1-e^{-(2k+\alpha+1) |\lambda|t}\big)(\sinh t|\lambda|)^{-s-1}\,dt.\]
Multiplying both sides by $ \frac{2|\lambda|^{s+\alpha+1}}{\Gamma(\alpha+1)}\widehat{f}(\lambda,k)\phi_{k,\lambda}^{\alpha}(x)$, we have
\begin{multline*}
\frac{\Gamma(1-s)}{s}\frac{2|\lambda|^{\alpha+1}}{\Gamma(\alpha+1)}(2|\lambda|)^{s}\frac{\Gamma\Big(\frac{2k+\alpha+1}{2}+\frac{1+s}{2}\Big)}{\Gamma\Big(\frac{2k+\alpha+1}{2}+\frac{1-s}{2}\Big)}\widehat{f}(\lambda,k)\phi_{k,\lambda}^{\alpha}(x)=c_1\frac{2|\lambda|^{\alpha+1}}{\Gamma(\alpha+1)}|\lambda|^{s}\widehat{f}(\lambda,k)\phi_{k,\lambda}^{\alpha}(x)\\
 +\frac{2|\lambda|^{\alpha+1}}{\Gamma(\alpha+1)}\int_0^{\infty}\big(1-e^{-(2k+\alpha+1) |\lambda|t}\big)\Big(\frac{t|\lambda|}{\sinh t\lambda}\Big)^{s+1}\widehat{f}(\lambda,k)\phi_{k,\lambda}^{\alpha}(x)t^{-s-1}\,dt.
\end{multline*}
Using \eqref{relation_convolution_transform}, \eqref{convolution into product} and summing over $k$, we obtain
\begin{multline}\label{equation 2}
\frac{\Gamma(1-s)}{s}\frac{2|\lambda|^{\alpha+1}}{\Gamma(\alpha+1)}(2|\lambda|)^{s}\sum_{k=0}^{\infty}
\frac{\Gamma\Big(\frac{2k+\alpha+1}{2}+\frac{1+s}{2}\Big)}
{\Gamma\Big(\frac{2k+\alpha+1}{2}+\frac{1-s}{2}\Big)}
\widehat{f}(\lambda,k)\phi_{k,\lambda}^{\alpha}(x)\\
=c_1|\lambda|^{s}f^{\lambda}(x)
 +\int_0^{\infty}\big(f^{\lambda}(x)-f^{\lambda}\ast_\lambda h_t^{\lambda}(x)\big)
 \Big(\frac{t\lambda}{\sinh t\lambda}\Big)^{s+1}t^{-s-1}\,dt.
\end{multline}
We now rewrite the last integral as a sum of the following two integrals:
$$ A =f^{\lambda}(x)\int_0^{\infty}\Big(\Big(\frac{t\lambda}{\sinh t\lambda}\Big)^{s+1}-1\Big)t^{-s-1}\,dt,$$
$$ B = \int_0^{\infty}\Big(f^{\lambda}(x)-\Big(\frac{t\lambda}{\sinh t\lambda}\Big)^{s+1}f^{\lambda}\ast_\lambda h^\lambda_t(x)\Big)t^{-s-1}\,dt.$$
Note that the first integral  $ A $ is  equal to
\[|\lambda|^sf^{\lambda}(x)\int_0^{\infty}\Big(\Big(\frac{t}{\sinh t}\Big)^{s+1}-1\Big)t^{-s-1}\,dt=:-c_2|\lambda|^sf^{\lambda}(x).\]
It happens that $c_1=c_2$. Indeed, 
\begin{align*}
c_1-c_2&=\int_0^{\infty}\big(\cosh t-1\big)(\sinh t)^{-s-1}\,dt+\int_0^{\infty}\Big(\Big(\frac{t}{\sinh t}\Big)^{s+1}-1\Big)t^{-s-1}\,dt\\
&=\int_0^{\infty}\big((\cosh t)(\sinh t)^{-s-1}-t^{-s-1}\big)\,dt.
\end{align*}
Consider the integral
\[\int_{\delta}^{\infty}(\cosh t)(\sinh t)^{-s-1}\,dt=\int_{\sinh \delta}^{\infty}t^{-s-1}\,dt=\int_{\delta}^{\infty}t^{-s-1}\,dt-\int_{\delta}^{\sinh \delta}t^{-s-1}\,dt.\]
This gives
\[\int_{\delta}^{\infty}\big((\cosh t)(\sinh t)^{-s-1}-t^{-s-1}\big)\,dt=-\int_{\delta}^{\sinh \delta}t^{-s-1}\,dt,\]
which converges to $0$ as $\delta\to0$.
Finally, using the expression of $B$, we multiply \eqref{equation 2} by $e^{-i\lambda w}$ and integrate over $\lambda$ variable to get
\[\mathcal{L}_s f(x,w) = \frac{s}{\Gamma(1-s)}\int_0^\infty (f(x,w)-f\ast K_{t,s}(x,w))t^{-s-1}\,dt.\]
Since $\frac{s}{\Gamma(1-s)}=\frac{1}{|\Gamma(-s)|}$, we obtain the desired representation of $\L_s$.
\end{proof}

We can modify the integral representation using the properties of $K_{t,s}$ (Proposition \ref{properties of modified kernel}):

\begin{proposition}
For $0<s<1$, $f\in W^{s,2}(X)$, we have \[\mathcal{L}_s f(\xi) = \frac{1}{\Gamma(-s)}\int_0^\infty\int_X (f(\xi)-f(\eta))T^\eta K_{t,s}(\xi)t^{-s-1}\,d\mu(\eta)\,dt.\]
\end{proposition}
\begin{proof}
From last theorem, we have \begin{align*}
\mathcal{L}_s f(\xi) &= \frac{1}{\Gamma(-s)}\int_0^\infty (f(\xi)-f\ast K_{t,s}(\xi))t^{-s-1}\,dt\\
&=\frac{1}{\Gamma(-s)}\int_0^\infty\left(\int_X f(\xi)T^\eta K_{t,s}(\xi)\,d\mu(\eta)-\int_XT^\eta K_{t,s}(\xi)f(\eta)\,d\mu(\eta)\right)t^{-s-1}\,dt\\
&=\frac{1}{\Gamma(-s)}\int_0^\infty\int_X (f(\xi)-f(\eta))T^\eta K_{t,s}(\xi)\,d\mu(\eta)t^{-s-1}\,dt.
\end{align*}
\end{proof}

\begin{proposition}
For $0<s<1$ and $f,g\in W^{s,2}(X)$, we have \[\langle\mathcal{L}_s f,g\rangle = \frac{1}{2\Gamma(-s)}\int_0^\infty \int_X\int_X (f(\xi)-f(\eta))\overline{(g(\xi)-g(\eta))} T^\eta K_{t,s}(\xi)\,d\mu(\eta)\,d\mu(\xi)\frac{dt}{t^{s+1}}.\]
\end{proposition}
\begin{proof}
We have \begin{align*}
\langle \mathcal{L}_s f,g\rangle &= \frac{1}{\Gamma(-s)}\int_X\left(\int_0^\infty\int_X (f(\xi)-f(\eta))\overline{g(\xi)}T^\eta K_{t,s}(\xi)\,d\mu(\eta)\,\frac{dt}{t^{s+1}}\right)d\mu(\xi)\\
&= \frac{1}{\Gamma(-s)}\int_X\left(\int_0^\infty\int_X (f(\eta)-f(\xi))\overline{g(\eta)}T^\xi K_{t,s}(\eta)\,d\mu(\xi)\,\frac{dt}{t^{s+1}}\right)d\mu(\eta)\\
&= -\frac{1}{\Gamma(-s)}\int_X\left(\int_0^\infty\int_X (f(\xi)-f(\eta))\overline{g(\eta)}T^{\eta\ast} K_{t,s}(\xi^\ast)\,d\mu(\xi)\,\frac{dt}{t^{s+1}}\right)d\mu(\eta),
\end{align*}
where $(x,w)^\ast = (x,-w)$. Using, $T^{\eta\ast} K_{t,s}(\xi^\ast) = T^\eta K_{t,s}(\xi)$ and Fubini's Theorem, we get: 
\begin{align*}
\langle \mathcal{L}_s f,g\rangle &=-\frac{1}{\Gamma(-s)}\int_X\left(\int_0^\infty \int_X (f(\xi)-f(\eta))\overline{g(\eta))} T^\eta K_{t,s}(\xi)\,d\mu(\eta)\,\frac{dt}{t^{s+1}}\right)d\mu(\xi).
\end{align*} Hence, we have 
\begin{align*}
\langle \mathcal{L}_s f,g\rangle &=\frac{1}{2\Gamma(-s)}\int_X\left(\int_0^\infty \int_X (f(\xi)-f(\eta))\overline{g(\xi)-g(\eta))} T^\eta K_{t,s}(\xi)\,d\mu(\eta)\,\frac{dt}{t^{s+1}}\right)d\mu(\xi).\\
&=\frac{1}{2\Gamma(-s)}\int_0^\infty\left(\int_X \int_X (f(\xi)-f(\eta))\overline{g(\xi)-g(\eta))} T^\eta K_{t,s}(\xi)\,d\mu(\eta)\,d\mu(\xi)\right)\frac{dt}{t^{s+1}}.
\end{align*}
\end{proof}

Finally, for $0<s<1$ and $\delta>0$, we define \textit{ground state representation} $\mathcal{H}[f]$ for $f\in W^{s,2}(X)$ by:\[\mathcal{H}[f] = \langle\mathcal{L}_s f,f\rangle - A_{\alpha,s}\int_X \frac{|f(x,w)|^2}{\left(\left(\delta+\frac{x^2}{2}\right)^2+w^2\right)^s},\] where \[A_{\alpha,s} = (4\delta)^s\left(\frac{\Gamma(\frac{\alpha+s+2}{2})}{\Gamma(\frac{\alpha-s+2}{2})}\right)^2.\]

\begin{proposition}
Let $0<s<1$, $\delta>0$ and $F\in C^\infty_{c}(X)$, that is $F$ is an infinitely differentiable function defined on $X$ with compact support. If we define $G(\xi) = \frac{F(\xi)}{u_{-s,\delta}(\xi)}$, then \[\mathcal{H}_s[F] = \frac{1}{2\Gamma(-s)}\int_0^\infty\int_X\int_X |G(\xi)-G(\eta)|^2 T^\eta K_{t,s}(\xi)u_{-s,\delta}(\xi)u_{-s,\delta}(\eta)\,d\mu(\eta)\,d\mu(\xi)\frac{dt}{t^{s+1}},\] where $u_{s,\delta}$ is defined in \eqref{function u}. 
\end{proposition}
\begin{proof}
In the previous proposition, if we take $g(\xi) = u_{-s,\delta}(\xi)$ and $f(\xi) = \frac{|F(\xi)|^2}{u_{-s,\delta}(\xi)}$, then we have $\langle\mathcal{L}_s f,g\rangle$ equals to \[\frac{1}{2\Gamma(-s)}\int_0^\infty \int_X\int_X \left(g(\xi)-g(\eta)\right) \left(\frac{F^2(\xi)}{g(\xi)}-\frac{F^2(\eta)}{g(\eta)}\right) T^\eta K_{t,s}(\xi)\,d\mu(\eta)\,d\mu(\xi)\frac{dt}{t^{s+1}}.\] On simplification, we get \[\frac{1}{2\Gamma(-s)}\int_0^\infty \int_X\int_X \left(|F(\xi)-F(\eta)|^2 - \left|\frac{F(\xi)}{g(\xi)}-\frac{F(\eta)}{g(\eta)}\right|^2g(\xi)g(\eta)\right) T^\eta K_{t,s}(\xi)\,d\mu(\eta)\,d\mu(\xi)\frac{dt}{t^{s+1}}.\] On the other hand, using Proposition \ref{main_proposition} and the fact that $\L_s$ is self-adjoint, we have $\langle\mathcal{L}_s f,g\rangle$ equals to \[(4\delta)^s\left(\frac{\Gamma(\frac{\alpha+s+2}{2})}{\Gamma(\frac{\alpha-s+2}{2})}\right)^2\int_X \frac{|F(\xi)|^2}{u_{-s,\delta}(\xi)}u_{s,\delta}(\xi)\,d\mu(\xi).\] Equating both, and noting that $u_{-s,\delta}(x,w)/u_{s,\delta}(x,w) = \left((\delta+x^2/2)^2+w^2\right)^{s}$, we get the desired result.
\end{proof}

Finally we prove Hardy inequality for fractional powers of generalized sublaplacian.
\begin{proof}[Proof of Theorem \ref{hardygeneralizedSublaplacian}]\label{proof of hardy sublaplacian}
Let $f\in C_c^\infty(X)$. Define $g(\xi) = f(\xi)/u_{-s,\delta}(\xi)$, where $u_{s,\delta}$ is defined in \eqref{function u}. Using the last Proposition and Fubini's Theorem, we have 
\begin{equation*}
\mathcal{H}_s[f] = \frac{1}{2\Gamma(-s)}\int_X\int_X \left(\int_0^\infty T^\eta K_{t,s}(\xi)\frac{dt}{t^{s+1}}\right)|g(\xi)-g(\eta)|^2 u_{-s,\delta}(\xi)u_{-s,\delta}(\eta)\,d\mu(\eta)\,d\mu(\xi).
\end{equation*}
Since the generalized translation operator is a positive operator (see \cite{stempak1}[Proposition 3.2]) and $\int_0^\infty K_{t,s}(\xi)t^{-s-1}\,dt\geqslant 0$ (Proposition \ref{explicit modified kernel}), we have for all $\xi,\eta\in X$ \[\int_0^\infty T^\eta K_{t,s}(\xi)\frac{dt}{t^{s+1}}\geqslant 0.\] Also the remaining terms in the expression of $\mathcal{H}_s[f]$ are positive for all $\xi$ and $\eta$. Therefore, we conclude $\mathcal{H}_s[f]\geqslant 0$ for all $f\in C_c^\infty(X)$. Hence, for $f\in C_c^\infty(X)$, we have \[\langle\L_s f, f\rangle\geqslant (4\delta)^s\left(\frac{\Gamma(\frac{\alpha+s+2}{2})}{\Gamma(\frac{\alpha-s+2}{2})}\right)^2\int_X \frac{|f(\xi)|^2}{\left(\left(\delta+\frac{x^2}{2}\right)^2+w^2\right)^s}\,d\mu(\xi).\]
Next, let $f\in W^{s,2}(X)$. Since $C_c^\infty(X)$ dense in $W^{s,2}(X)$, therefore there exists a sequence $\{f_j\}$ with each $f_j \in C_c^\infty(X)$ such that $f_j\rightarrow f$ in $W^{s,2}(X)$. Passing to a sub-sequence, we can assume $f_j\rightarrow f$ pointwise a.e.. The continuity of inner-product on $W^{s,2}(X)$ implies $\langle \L_s f_j,f_j\rangle\rightarrow \langle \L_s f,f\rangle$. On the other hand, the inequality $|f(x,w)|^2/((\delta+x^2)^2+w^2)^s\leqslant \delta^{-2s}|f(x,w)|^2$ together with Dominated convergence theorem, implies \[\int_X \frac{|f_j(\xi)|^2}{\left(\left(\delta+\frac{x^2}{2}\right)^2+w^2\right)^s}\,d\mu(\xi)\rightarrow\int_X \frac{|f(\xi)|^2}{\left(\left(\delta+\frac{x^2}{2}\right)^2+w^2\right)^s}\,d\mu(\xi).\] From this we conclude that the inequality holds for all $f\in W^{s,2}(X)$.

Finally, we note that the both sides of the inequality are equal for $f = u_{-s,\delta}$. Hence the constant involved in the inequality is sharp.
\end{proof}

\section{Proof of the main Theorem}

We recall from Section \ref{Spherical_harmonics} that for $j= 1,2,\cdots,a_m$ and $m= 0,1,2,\cdots$, $\{Y_{m,j}\}$ forms orthonormal basis for $L^2(\mathbb{S}^{n-1})$. Corresponding to each spherical harmonic $Y_{m,j}$, we define solid harmonics $P_{m,j}$ on $\R^n$ by \[P_{m,j}(x) = |x|^{m}Y_{m,j}(x/|x|).\] Moreover, for $f\in L^2(\R^{n+1})$, we have \[f(x,w) = \sum_{m=0}^\infty\sum_{j=1}^{a_m}f_{m,j}(|x|,w)P_{m,j}(x),\]
where $f_{m,j}(|x|,w) = \int_{\mathbb{S}^{n-1}}f(|x|\omega,w)P_{m,j}(|x|\omega)\,d\omega$.\\

Suppose $f\in W^{s,2}(H)$, such that $f(x,w) = g(|x|,w)P(x)$, where $P$ is a solid harmonics of degree $m$. Then using spectral decomposition and the Hecke Bockner formula (Prosposition \ref{Hermite_projection_heckhe_bockner}), we have \begin{align*}
\G_s f(x,w) &=\frac{1}{2\pi}\int_\R \sum_{k=0}^\infty (2|\lambda|)^s\frac{\Gamma(\frac{2k+n}{4}+\frac{1+s}{2})}{\Gamma(\frac{2k+n}{4}+\frac{1-s}{2})} \P_k(\lambda)f^\lambda(x)e^{-i\lambda w}\,d\lambda\\
&=\int_\R \sum_{k=0}^\infty \left( d^n_{m,k}(s) A^\lambda_{k,m}(g)\phi^{n/2+m-1}_{k,\lambda}(|x|)\right)P(x)|\lambda|^{n/2+m+s}e^{-i\lambda w}\,d\lambda,
\end{align*} where \[d^n_{m,k}(s)= \frac{2^{s+1}}{2\pi}\frac{\Gamma(\frac{2k+n/2+m+1+s}{2})}{\Gamma(\frac{2k+n/2+m+1-s}{2})}\frac{\Gamma(k+1)}{\Gamma(k+n/2+m)},\] and \[A^\lambda_{k,m}(g) = \int_{0}^\infty g^\lambda(r)\phi^{n/2+m-1}_{k,\lambda}(r)\,r^{n+2m-1}\,dr.\] Therefore, using the orthogonality of solid harmonics (with respect to inner product inherited form $L^2(\mathbb{S}^{n-1})$), we have \[\langle\mathcal{G}_s f,f\rangle = \int_\R \sum_{k=0}^\infty d^n_{m,k}(s)|A^\lambda_{k,m}(g)|^2|\lambda|^{\frac{n}{2}+m+s}\,d\lambda.\] Moreover, treating $g$ as a function on $X$, we have \[\langle \mathcal{L}_s g, g\rangle = \frac{2^{s+1}}{2\pi} \int_\R\sum_{k=0}^\infty\frac{\Gamma(\frac{2k+\alpha+2+s}{2})}{\Gamma(\frac{2k+\alpha+2-s}{2})}\frac{\Gamma(k+1)}{\Gamma(k+\alpha+1)}|B^\lambda_{k,\alpha}(g)|^2|\lambda|^{\alpha+1+s}\,d\lambda,\] where $B^\lambda_{k,\alpha}(g) = \int_0^\infty g^\lambda(x)\phi^\alpha_{k,\lambda}(x)x^{2\alpha+1}\,dx$. So for $\alpha = \frac{n}{2}+m-1$, we have 
\begin{equation}\label{grushin sublaplacian equality}
\langle\mathcal{G}_s f,f\rangle = \langle \mathcal{L}_s g, g\rangle.
\end{equation}

Now let $f\in C_c^\infty(H)$. Since $C_c^\infty(H)$ dense in $W^{s,2}(H)$, we have \[f(x,w) = \sum_{m=0}^\infty\sum_{j=1}^{a_m}f_{m,j}(|x|,w)P_{m,j}(x),\] where $f_{m,j}(|x|,w) = \int_{\mathbb{S}^{n-1}}f(|x|\omega,w)P_{m,j}(|x|\omega)\,d\omega$. Using \eqref{grushin sublaplacian equality}, we have \begin{align*}
\langle\mathcal{G}_s f,f\rangle &= \sum_{m=0}^\infty\sum_{j=1}^{a_m}\langle\mathcal{G}_s (f_{m,j}P_{m,j}),f_{m,j}P_{m,j}\rangle\\
&= \sum_{m=0}^\infty\sum_{j=1}^{a_m}\langle\mathcal{L}_s f_{m,j},f_{m,j}\rangle.
\end{align*}
Using the Theorem \ref{hardygeneralizedSublaplacian}, we have
\begin{align*}
\langle\mathcal{G}_s f,f\rangle &\geqslant \sum_{m=0}^\infty\sum_{j=1}^{a_m}(4\delta)^s\left(\frac{\Gamma(\frac{n/2+m+s+1}{2})}{\Gamma(\frac{n/2+m-s+1}{2})}\right)^2\int_\R\int_0^\infty \frac{|f_{m,j}(x,w)|^2x^{n+2m-1}}{\left(\left(\delta+\frac{x^2}{2}\right)^2+w^2\right)^s}\,dx\,dw\\
&\geqslant\inf_{m\geqslant 0}\left\{(4\delta)^s\left(\frac{\Gamma(\frac{n/2+m+s+1}{2})}{\Gamma(\frac{n/2+m-s+1}{2})}\right)^2\right\}\sum_{m=0}^\infty\sum_{j=1}^{a_m}\int_\R\int_0^\infty \frac{|f_{m,j}(x,w)|^2x^{n+2m-1}}{\left(\left(\delta+\frac{|x|^2}{2}\right)^2+w^2\right)^s}\,dx\,dw.
\end{align*}
Moreover, we have \[\int_{\R^n} |f(x,w)|^2\,dx = \sum_{m=0}^\infty\sum_{j=1}^{a_m} \int_0^\infty|f_{m,j}(r,w)|^2\,r^{n+2m-1}\,dr,\] and \[\inf_{m\geqslant 0}\left\{(4\delta)^s\left(\frac{\Gamma(\frac{n/2+m+s+1}{2})}{\Gamma(\frac{n/2+m-s+1}{2})}\right)^2\right\} = (4\delta)^s\left(\frac{\Gamma(\frac{n/2+s+1}{2})}{\Gamma(\frac{n/2-s+1}{2})}\right)^2.\]
Therefore, \begin{align*}
\langle\mathcal{G}_s f,f\rangle\geqslant(4\delta)^s\left(\frac{\Gamma(\frac{n/2+s+1}{2})}{\Gamma(\frac{n/2-s+1}{2})}\right)^2\int_{\R}\int_{\R^n} \frac{|f(x,w)|^2}{\left(\left(\delta+\frac{|x|^2}{2}\right)^2+w^2\right)^s}\,dx\,dw.
\end{align*}
Once we proved the inequality for $f\in C_c^\infty(H)$, we argue as in the proof of Theorem \ref{proof of hardy sublaplacian} to conclude that the inequality holds true for $f\in W^{s,2}(H)$.\\
Finally, we show that the constants involved in the inequality are sharp. For $-1<s<1$ and $\delta>0$, define $v_{s,\delta}$ on $H$ by \[v_{s,\delta}(x,w) = \left((\delta+|x|^2/2)^2+w^2\right)^{-\frac{n/2+1+s}{2}}.\] Using Proposition \ref{main_proposition},we have \begin{align*}
\langle \G_s v_{-s,\delta}, v_{-s,\delta}\rangle &= \langle \L_s v_{-s,\delta}, v_{-s,\delta}\rangle\\
&= (4\delta)^s\frac{\Gamma^2(\frac{n/2+s+1}{2})}{\Gamma^2(\frac{n/2-s+1}{2})}\langle v_{s,\delta},v_{-s,\delta}\rangle\\
&= (4\delta)^s\frac{\Gamma^2(\frac{n/2+s+1}{2})}{\Gamma^2(\frac{n/2-s+1}{2})}\int_\R\int_0^\infty \frac{x^{n-1}dx\,dw}{\left(\left(\delta+\frac{x^2}{2}\right)^2+w^2\right)^{n/2+1}}\\
&= (4\delta)^s\frac{\Gamma^2(\frac{n/2+s+1}{2})}{\Gamma^2(\frac{n/2-s+1}{2})}\int_{\R}\int_{\R^n} \frac{|u_{-s,\delta}(x,w)|^2dx\,dw}{\left(\left(\delta+\frac{x^2}{2}\right)^2+w^2\right)^{s}}.\\
\end{align*} Therefore, equality is achieved  for $f = u_{-s,\delta}$. Hence the constants involved in the inequality are sharp.

\end{document}